\documentclass[11pt,a4paper]{article}

	\usepackage[left=2.45cm, top=2.45cm,bottom=2.45cm,right=2.45cm]{geometry}
	\usepackage{hyperref}
	\usepackage{tikz}
	\usetikzlibrary{arrows.meta}
	\usepackage{mathtools,amssymb,amsthm,mathrsfs,calc,graphicx,xcolor,dsfont,tikz,pgfplots,bm, extarrows}
	\usepackage[british]{babel}
	\usepackage{amsfonts}             

	\usepackage[T1]{fontenc}
	\usepackage{aligned-overset}
	\numberwithin{equation}{section}

	\newtheorem {theorem}{Theorem}[section]
	
	\newtheorem {lemma}[theorem]{Lemma}
	\newtheorem {corollary}[theorem]{Corollary}
	
	{\theoremstyle{definition}
		
	}
	{\theoremstyle{theorem}

	}

	\newcommand{\one}{\mathbf{1}}
	\newcommand{\E}{\mathbf{E}}
	\newcommand{\p}{\mathbf{P}}
	
	\newcommand{\N}{\mathbb{N}} 
	\newcommand{\R}{\mathbb{R}}

	\usepackage{enumitem}
	
	\def\cA{\mathcal{A}}
	
	\def\cC{\mathcal{C}}

	\def\cL{\mathcal{L}}

	\def\cS{\mathcal{S}}

	\def\cW{\mathcal{W}}

	\makeatletter
	\let\@fnsymbol\@alph
	\makeatother
	\allowdisplaybreaks
\begin{document}

\title{\bfseries Poisson approximation for cycles in the generalised random graph}

\author{Matthias Lienau\footnotemark[1]}

\date{}
\renewcommand{\thefootnote}{\fnsymbol{footnote}}

\footnotetext[1]{Hamburg University of Technology. Email: matthias.lienau@tuhh.de}

\maketitle

\begin{abstract}
	The generalised random graph contains $n$ vertices with positive i.i.d.\ weights. The probability of adding an edge between two vertices is increasing in their weights. We require the weight distribution to have finite second moments and study the point process $\cC_n$ on $\{3,4,\dots\}$, which counts how many cycles of the respective length are present in the graph. We establish convergence of $\cC_n$ to a Poisson point process. Under the stronger assumption of the weights having finite fourth moments we provide the following results. When $\cC_n$ is evaluated on a bounded set $A$, we provide a rate of convergence. If the graph is additionally subcritical, we extend this to unbounded sets $A$ at the cost of a slower rate of convergence. From this we deduce the limiting distribution of the length of the shortest and the longest cycle when the graph is subcritical, including rates of convergence. All mentioned results also apply to the Chung-Lu model and the Norros-Reittu model.	
\end{abstract}

\section{Introduction}

So-called complex networks are ubiquitous in our modern world. This term refers to very large networks having a difficult global structure, such as the brain with its neurons and the connections between them. Further examples include social networks, the internet or (tele-)communication networks, showing that complex networks are not restricted to a niche research area but appear in many different fields of study. Due to the immense size of complex networks, a direct analysis is usually not feasible. Instead, one is looking for probabilistic models capturing key properties of real-world networks. We refer the reader to \cite{vdH_book} for more details on complex networks. One of the aforementioned desired properties is referred to as being scale-free, see also \cite{welldone}: It turns out that the proportion of vertices having degree $k$ is more or less proportional to $k^{-\gamma}$ for some $\gamma>1$ in various real-world examples. Visually speaking, when plotting the degree distribution on a $\log-\log$ scale, one obtains the shape of a line. Emulating this property with random graphs thus requires models with suitable degree distributions. Among these models is the generalised random graph introduced in \cite{GenRandGraph}, which can be seen as a generalisation of the classical Erd\H{o}s-R\'enyi graph. We work with a slight adaption of the original model, which was also used in \cite{vdEsker}. For $n\in\N$ one takes the vertices $[n]=\{1,\dots,n\}$ and assigns i.i.d.\ copies $W_1,\dots,W_n$ of a positive random variable $W$ to them. Conditionally on these weights, we connect any two distinct vertices $i,j\in[n]$, indicated by $i\leftrightarrow j$, independently with probability
\begin{align*}
	\p_\cW(i\leftrightarrow j)=\frac{W_iW_j}{W_iW_j+L_n}\quad\text{for}\quad L_n=\sum_{i=1}^nW_i,
\end{align*}
where $\p_\cW$ denotes the conditional probability with respect to $W_1,\dots,W_n$. Intuitively, one wants to use $W_iW_j/L_n$ as connection probability but needs an adjustment as this is not necessarily bounded by one. A suitable weight distribution leads to the desired scale-free behaviour as shown in Theorem 3.1 in \cite{GenRandGraph}. There are related random graph models such as the Chung-Lu model \cite{ChungLu2} or the Norros-Reittu model \cite{NorrosReittu2006} which use other variations of the desired connection probability $W_iW_j/L_n$. In Example 3.6 in \cite{Janson2008} it has been shown that these models are asymptotically equivalent if $\p(W>t)=o(t^{-2})$. They are also referred to as rank-1 inhomogeneous random graphs and Section 16.4 in \cite{BolJanRio2007} provides information on the emergence of a giant component, i.e.\ a component containing a positive fraction of all vertices as $n\to\infty$. It turns out that the graph is supercritical (with giant component) if and only if $\E[W^2]>\E[W]$. It is called subcritical when $\E[W^2]<\E[W]$ and critical when $\E[W^2]=\E[W]$. The exact size of the largest component in the subcritical regime was further specified in \cite{Janson2008_comp_size} whereas \cite{vdHBhavLe2010} and \cite{vdHBhavLe2012} studied the component size at criticality. Some of our results are limited to the subcritical regime.

Our aim is to analyse the number of cycles in the generalised random graph. Cycle counts and more general subgraph counts were already studied extensively for different classes of graphs such as the Erd\H{o}s-R\'enyi graph and random regular graphs, see e.g.\ Chapter 4 and Section 2.4, respectively, in \cite{BollobasBook} and the references therein. Recently, there have been results for the expected number of cycles of a fixed length in the generalised random graph when the weight distribution has a regularly varying tail with parameter $\beta\in(1,2)$, i.e.\ when $\p(W>t)=\ell(t)t^{-\beta}$ for a slowly varying function $\ell$. In \cite{Janssen2019}, the authors study asymptotics of the expected number of cycles (and cliques). This builds on results of \cite{bianconi2005}, which required a cutoff for the weight distribution. Additionally, \cite{vdh_counting_triangles} investigates the asymptotic expected number of triangles and shows that most of the triangles are found among vertices of degree $\sqrt{n}$. In \cite{clara_optimal_subgraph} a similar approach is employed to investigate the asymptotic expected number of more general subgraphs, including cycles of any length. The aforementioned results are all in the case where $\E[W^2]=\infty$, so in particular in the supercritical regime. 

There are also results for weights with lighter tails. Throughout the paper we write
\begin{align}
	\lambda_k=\frac{1}{2k}\bigg(\frac{\E[W^2]}{\E[W]}\bigg)^k \label{eq:def_lambda_k}
\end{align}
for $k\in\N$. In \cite{liu_dong} it is shown that the number of triangles converges to a Poisson random variable with parameter $\lambda_3$ when the weights are bounded. This result was generalised in \cite{bobkov}, where it was shown that the number of cycles having length $k$ converges to a Poisson random variable with parameter $\lambda_k$, given that $\E[W^{2k+1}]<\infty$ or a slightly weaker tail condition holds. Additionally, they provide the rate of convergence $O(n^{-1/2})$ in the total variation distance. In this paper we strengthen this result in different ways. If one is only interested in a qualitative result, we manage to reduce the moment condition to $\E[W^2]<\infty$. In order to keep $\lambda_k$ finite, this cannot be relaxed any further. Concerning quantitative results, we reduce the moment condition $\E[W^{2k+1}]<\infty$ in \cite{bobkov} to $\E[W^4]<\infty$ by using a similar strategy but avoiding certain summands. Having finite fourth moments no longer depends on the cycle length $k$ and is a weaker assumption even for the shortest cycle of length $k=3$. Additionally, we increase the rate of convergence to $O(n^{-1})$. Finally, we do not restrict to cycles of a fixed length $k\in\{3,4,\dots\}$ but allow for lengths in some bounded set $A\subseteq\{3,4,\dots\}$ and establish asymptotic independence of the numbers of cycles of different lengths. We have additional results in the subcritical regime, i.e.\ when $\E[W^2]<\E[W]$. In this case, we show that with high probability there are no cycles whose length grows at least logarithmically in $n$. Moreover, we derive quantitative Poisson convergence for the number of all cycles whose lengths lie in some, possibly unbounded, set $A\subseteq\{3,4,\dots\}$. This allows us to derive the limiting distribution of and a rate of convergence in the Kolmogorov distance for the length of the shortest and the longest cycle. 

The rest of this paper is organised as follows. The previously mentioned results are presented in Section \ref{section:main} where we also discuss how they transfer to the Chung-Lu model and the Norros-Reittu model. In Section \ref{sec:setting} we provide the main tools which we employ whereas the proofs are postponed to the remaining sections.

\section{Main results}\label{section:main}
We consider the point process $\cC_n$ on $\N_{\geq3}=\{3,4,\dots\}$ where $\cC_n(A)$ denotes the number of cycles whose length lies in $A$ for $A\subseteq\N_{\geq3}$. For $k\in\N$ and $x_1,\dots,x_k\in\N_{\geq3}$ we write $\cC_n(x_1,\dots,x_k)$ instead of $\cC_n(\{x_1,\dots,x_k\})$ to simplify notation. Let $\eta$ denote a Poisson point process on $\N_{\geq3}$ with intensity measure given by $\lambda(\{k\})=\lambda_k$ for $k\in\N_{\geq3}$ and $\lambda_k$ as in \eqref{eq:def_lambda_k}. Our first theorem deals with convergence of $\cC_n$ in distribution  with respect to the vague topology on the set of counting measures on $\N_{\geq3}$, denoted by $\overset{d}{\longrightarrow}$. This is equivalent to convergence of the finite-dimensional distributions of $(\cC_n(k))_{k\in\N_{\geq3}}$, see Chapter 4 in \cite{Kallenberg_random_measures} for more details on $\overset{d}{\longrightarrow}$ in this setting. 
\begin{theorem}\label{thm:vague_conv}
	If $\E[W^2]<\infty$, then
	\begin{align*}
		\cC_n\overset{d}{\longrightarrow} \eta \quad\text{as }n\to\infty. 
	\end{align*}
\end{theorem}
From this we obtain that the numbers of cycles having different lengths are asymptotically independent. This follows from the aforementioned equivalence to convergence of the finite-dimensional distributions. Under stronger assumptions we achieve rates of convergence for $\cC_n(A)$ in the total variation distance. For two $\N_0$-valued random variables $X$ and $Y$ this distance is given by
\begin{align*}
	d_{TV}(X,Y)=\sup_{A\subseteq\N_0}|\p(X\in A)-\p(Y\in A)|.
\end{align*}
\begin{theorem}\label{thm:conv_rate_bounded}
	If $\E[W^4]<\infty$, then for all $N\in\N$ there exists a constant $C>0$ such that for all $A\subseteq\{3,4,\dots,N\}$ and $n\geq3$,
	\begin{align*}
		d_{TV}(\cC_n(A),\eta(A))\leq\frac{C}{n}.
	\end{align*}		
\end{theorem}
We prove this theorem with typical Poisson approximation methods, which we describe in more detail in Section \ref{sec:setting}. Our remaining results concern the case $\E[W^2]<\E[W]$, i.e.\ the subcritical regime of the graph. Note that this assumption implies that the intensity measure $\lambda$ of the limiting process $\eta$ is finite. This allows us to treat cases where the considered cycle lengths  do not need to stay bounded with increasing $n$.  The following theorem shows that there are asymptotically no cycles whose lengths grow at least logarithmically in $n$, as $n\to\infty$. Throughout this paper we denote by $\lfloor \cdot \rfloor$ the floor function.

\begin{theorem}\label{thm:no_long_cycles}
	Let $\E[W^2]<\E[W]$ and $a>0$. There exists a constant $C>0$ such that for all $n\geq3$,
	\begin{align*}
		\p(\cC_n(\lfloor a\log(n) \rfloor+1,\dots,n)>0)\leq Cn^{a\log(p)}\quad \text{for}\quad p=\frac{2\E[W^2]}{\E[W^2]+\E[W]}.
	\end{align*}
\end{theorem}
Note that $p<1$ and thus $\log(p)<0$ in the previous theorem so that the upper bound tends to zero as $n\to\infty$. The following theorem is similar to the bounded case in Theorem \ref{thm:conv_rate_bounded}. At the cost of a slower rate of convergence the set $A$ no longer needs to be bounded. The proof relies on the fact that \textit{long} cycles can be ignored, as they are unlikely due to Theorem \ref{thm:no_long_cycles}. For the remaining cycles, we use the same approach as for Theorem \ref{thm:conv_rate_bounded}.

\begin{theorem}\label{thm:conv_rate_unbounded}
	Let $\E[W^4]<\infty$ and $\E[W^2]<\E[W]$. There exists a constant $C>0$ such that for all $n\geq3$ and $A\subseteq\N_{\geq3}$,
	\begin{align*}
		d_{TV}(\cC_n(A),\eta(A))\leq\frac{C\log(n)^3}{n}.
	\end{align*}
\end{theorem}
The previous theorem allows us to derive the asymptotic distribution of the length of the shortest cycle $\cC_{\min}^{(n)}$ and of the longest cycle $\cC_{\max}^{(n)}$. When there are no cycles at all, we choose the convention $\cC_{\max}^{(n)}=\cC_{\min}^{(n)}=0$. We define the $\{0,3,4,\dots\}-$valued random variables $\cS$ and $\cL$ for the limits of the shortest and longest cycle, respectively, via 
\begin{align*}
	\p(\cS=0)=\p(\cL=0)=\exp\bigg(-\sum_{k=3}^\infty\lambda_k\bigg)
\end{align*}
and 
\begin{align*}
	\p(3\leq \cS \leq t)&=1-\exp\bigg(-\sum_{k=3}^t\lambda_k\bigg)
\end{align*}
as well as
\begin{align*}
	\p(\cL \leq t)=\exp\bigg(-\sum_{k=t+1}^{\infty}\lambda_k\bigg)
\end{align*}
for all $t\in\N_{\geq3}$.	We measure the rate of convergence in the Kolmogorov distance given by
\begin{align*}
	d_{Kol}(X,Y)=\sup_{t\in \R}|\p(X\leq t)-\p(Y\leq t)|
\end{align*}
for two random variables $X$ and $Y$.
\begin{theorem}\label{thm:shortest_longest_cycle}
	Let $\E[W^4]<\infty$ and $\E[W^2]<\E[W]$.
	\begin{enumerate}
		\item There exists a constant $C>0$ such that for all $n\geq3$,
		\begin{align*}
			d_{Kol}(\cC_{\min}^{(n)},\cS )\leq\frac{C\log(n)^3}{n}. 
		\end{align*}
		\item There exists a constant $C>0$ such that for all $n\geq3$,
		\begin{align*}
			d_{Kol}(\cC_{\max}^{(n)},\cL )\leq \frac{C\log(n)^3}{n}.
		\end{align*}
	\end{enumerate}
	
\end{theorem}
\subsection*{Relation with similar models}
The four theorems above also hold for the Chung-Lu model \cite{ChungLu2} and the (erased) Norros-Reittu model \cite{NorrosReittu2006}. The original Norros-Reittu model may produce loops, i.e.\ edges from a vertex to itself, and multiple edges between two vertices. In order to obtain a simple graph one erases all loops and merges multiple edges into single edges. This yields the erased Norros-Reittu model. We quickly define the models. Similarly to the generalised random graph, one takes the vertices $[n]=\{1,\dots,n\}$ equipped with i.i.d.\ weights $W_1,\dots,W_n$ for $n\in\N$. Conditionally on the weights, one connects any two distinct vertices $i,j\in[n]$ independently with probability
$$
 \p_{\cW,CL}(i\leftrightarrow j)=\min(1,W_iW_j/L_n) \quad\text{and}\quad \p_{\cW,NR}(i\leftrightarrow j)=1-\exp(-W_iW_j/L_n)
$$
in the Chung-Lu model and the erased Norros-Reittu model, respectively. If one is only interested in the qualitative results, they follow immediately from the asymptotic equivalence provided in Example 3.6 in \cite{Janson2008}. If one wishes to keep quantitative results, one would need to quantify this asymptotic equivalence. We provide a simpler argument. All our proofs rely on lower and upper bounds of the connection probability
$$\p_\cW(i\leftrightarrow j)=\frac{W_iW_j}{W_iW_j+L_n}$$ for two distinct vertices $i,j\in[n]$ in the generalised random graph. From $1+x\leq \exp(x)$ for all $x\in\R$ one obtains 
$$
\exp(-x)\leq \frac{1}{1+x} \Leftrightarrow \frac{x}{1+x}=1-\frac{1}{1+x}\leq 1-\exp(-x)
$$
and with $x=W_iW_j/L_n$ we conclude
\begin{align*}
	\frac{W_iW_j}{W_iW_j+L_n}\leq 1-\exp\bigg(-\frac{W_iW_j}{L_n}\bigg)\leq \min\bigg(1,\frac{W_iW_j}{L_n}\bigg),
\end{align*}
where the last inequality uses $\exp(x)>0$ and $1-\exp(-x)\leq x$ for all $x\in\R$. Therefore we get
\begin{align*}
	\p_\cW(i\leftrightarrow j)\leq \p_{\cW,NR}(i\leftrightarrow j)\leq \p_{\cW,CL}(i\leftrightarrow j).
\end{align*}
This is also known as stochastic domination of the underlying graphs, see e.g.\ Subsection 6.8.2 in \cite{vdH_book}. The inequalities above imply that the lower bound we use for the generalised random graph applies to all three models. The upper bounds we use are $1$ and $W_iW_j/L_n$ and therefore also remain upper bounds for all three models. Thus, the proofs and the results transfer to the Chung-Lu model and the Norros-Reittu model.

\section{Results on Poisson approximation and the Poisson Cram\'er-Wold device}\label{sec:setting}

We start by providing the framework of the Poisson approximation from \cite{twomomentssuffice}, but in a conditioned setting. Let $I$ denote an arbitrary index set and $\cA$ a $\sigma$-field. We write $\E_\cA$ as a short-hand notation for the conditional expectation with respect to $\cA$. For all $\alpha\in I$ let $X_\alpha$ be a Bernoulli random variable. We are interested in the number of successes $S=\sum_{\alpha\in I}X_\alpha$. Additionally, let $B_\alpha\subseteq I$ be such that $\alpha\in B_\alpha$ for all $\alpha\in I$ and define
\begin{align}
	b_1&= \sum_{\alpha\in I}\sum_{\beta\in B_\alpha}p_\alpha p_\beta\quad& & \text{for}\quad p_\alpha=\E_\cA[X_\alpha],\label{def:b1}\\
	b_2&= \sum_{\alpha\in I}\sum_{\beta\in B_\alpha\setminus\{\alpha\}}p_{\alpha\beta}\quad& &\text{for}\quad p_{\alpha\beta}=\E_\cA[X_\alpha X_\beta]\quad \text{and}\label{def:b2}\\
	b_3&= \sum_{\alpha\in I}\E_\cA\left[\left|\E[X_\alpha-p_\alpha|\sigma(\cA,X_\beta\colon \beta\in B_\alpha^c)]\right|\right].\nonumber
\end{align}
We denote a mixed Poisson distribution by $X\sim\mathrm{Poi}(Y)$, meaning that the distribution of $X$ conditionally on $Y$ is a Poisson distribution with parameter $Y$. The proof of the following lemma is analogue to the proof of Theorem 1 in \cite{twomomentssuffice} and is postponed to Section \ref{sec:poisson_lemma}. In the case that $\cA$ is the trivial $\sigma$-field, the lemma below follows from Theorem 1 in \cite{twomomentssuffice}. See also \cite{ross} and \cite{steinsmethod} for more details on Stein's method for Poisson approximation. 
\begin{lemma}\label{lem:two_mom_suff_conditioned_version}
	If $T_\cA\sim\mathrm{Poi}(\E_\cA[S])$, then
	\begin{align*}
		d_{TV}(S,T_\cA)\leq \E\big[\min\big(1,b_{1}+b_{2}+b_{3}\big)\big].
	\end{align*}
\end{lemma}
From this lemma we deduce the following statement. Its short proof is also in Section \ref{sec:poisson_lemma}.
\begin{lemma}\label{lem:Tv_bound_both_terms}
	For $Z\sim\mathrm{Poi}(\mu)$ with $\mu>0$,
	\begin{align*}
		d_{TV}(S,Z)\leq \E[\min(1,b_1+b_2+b_3)]+\E[\min(1,|\E_\cA[S]-\mu|)].
	\end{align*}
\end{lemma}
We wish to add a brief comment on the benefit of having the conditional version. Here it suffices to show $b_1+b_2+b_3\to0$ and $|\E_\cA[S]-\mu|\to 0$ in probability as $n\to\infty$ to obtain a qualitative, albeit not quantitative result. This allows us to get rid of the additional assumption $\E[W^4]<\infty$ required for the quantitative result in Theorem \ref{thm:conv_rate_bounded}.

We present the choices for $I,X_\alpha$ and $B_\alpha$ in our setting. For any set $D$ and $k\in\N$ we write $D^k_{\neq}$ for the $k$-tuples with pairwise distinct entries from $D$. For $k\in\N_{\geq3}$ and $n\in\N$ let 
\begin{align*}
	I_k=I_k(n)=\{\alpha\in[n]^k_{\neq}\colon \alpha_1=\min_{i=1,\dots,k}\alpha_i,\alpha_2<\alpha_k\},
\end{align*}
which represents the possible cycles of length $k$. The constraints on $\alpha$ in $I_k$ correspond to fixing a starting vertex $\alpha_1$ and an orientation of the cycle $\alpha$ to ensure that each cycle corresponds to exactly one $\alpha\in I_k$. We suppress the dependence of $I_k$ on $n$ in the notation for simplicity. In this work we study $\cC_n(A)$, the number of cycles whose length lies in some set $A\subseteq\N_{\geq3}$. For $k\in\N_{\geq3}$ and $\alpha\in I_k$ we write		
\begin{align}
	X_\alpha=\one\{\text{The cycle }\alpha\text{ exists} \}=\one\{\alpha_1\leftrightarrow\alpha_2\leftrightarrow\hdots\leftrightarrow \alpha_k\leftrightarrow\alpha_1\} \label{eq:def_X_alpha}
\end{align}
so that we have		
\begin{align*}
	\cC_n(A)=\sum_{k\in A}\sum_{\alpha\in I_k}X_\alpha=\sum_{\alpha\in I}X_\alpha\quad\text{for}\quad I=\bigcup_{k\in A}I_k.
\end{align*}
		
For $A\subseteq \N_{\geq3}$ and $\alpha\in I$ the choice of $B_\alpha$ will differ in Section \ref{sec:qualitative_theorem} and Section \ref{sec:quantitative}. The reason is a different choice of the $\sigma$-field $\cA$. In Section \ref{sec:qualitative_theorem}, i.e.\ for Theorem \ref{thm:vague_conv}, we choose $\cA=\cW$. Conditionally on the weights, the random variables $X_\beta,\beta\in I$, are independent as soon as they share no edges. Choosing
\begin{align}
	B_\alpha=B_\alpha(n)=\{\beta\in I\colon \alpha \text{ and }\beta \text{ share at least one edge}\} \label{def:B_alpha_conditioned}
\end{align}
therefore ensures that $b_3=0$ for this choice of $\cA$. In Section \ref{sec:quantitative} we choose $\cA$ as the trivial $\sigma$-field instead. If we were to choose $B_\alpha$ as above, then $b_3$ no longer needs to be zero as $X_\alpha$ and $X_\beta$ corresponding to two cycles with at least one shared vertex are not independent due to the weight of that vertex. This is why we choose
\begin{align}
	B_{\alpha}=B_\alpha(n)=\{\beta\in I\colon \alpha \text{ and }\beta \text{ share at least one vertex}\} \label{def:B_alpha}
\end{align}		
instead, which once more results in $b_3=0$. 

Next, we provide a corollary to the so-called Poisson Cramér-Wold device. We use this tool to establish the finite-dimensional convergence equivalent to the claim in Theorem \ref{thm:vague_conv}. A two-dimensional version of the statement can be found in Corollary 2.3 in \cite{poissoncramerwold}. The two-dimensional proof generalises to higher dimensions. 

For $m\in\N$ and $\N_0$-valued random variables $Y_1,\dots,Y_m$ as well as $q_1,\dots,q_m\in[0,1]$ let
\begin{align*}
	Y_i^{(q_i)}\sim\mathrm{Bin}(Y_i,q_i) \quad \text{ for }i=1,\dots,m, 
\end{align*}
where $Y_1^{(q_1)},\dots,Y_m^{(q_m)}$ are independent, conditionally on $Y_1,\dots,Y_m$. One refers to this as thinning.

\begin{corollary}\label{cor:cramer_wold}
	Let $(\boldsymbol{Y}_n=(Y_{1,n},\dots,Y_{m,n}))_{n\in\N}$ be a sequence of $\N_0^m$-valued random variables. Suppose that there are $\mu_1,\dots,\mu_m>0$ such that for all $q=(q_1,\dots,q_m)\in[0,1]^m,$
	\begin{align*}
		\sum_{i=1}^mY_{i,n}^{(q_i)}\overset{d}{\longrightarrow} Y_{q,\mu}\sim\mathrm{Poi}\bigg(\sum_{i=1}^mq_i\mu_i\bigg)\quad \text{as}\quad n\to\infty.
	\end{align*}
	Then 
	\begin{align*}
		\boldsymbol{Y}_n\overset{d}{\longrightarrow}\boldsymbol{Y}\sim\otimes_{i=1}^m\mathrm{Poi}(\mu_i) \quad \text{as}\quad n\to\infty.
	\end{align*}
\end{corollary}

\section{Proof of Theorem \ref{thm:vague_conv}}\label{sec:qualitative_theorem}
The following lemma is a direct consequence of the Marcinkiewicz-Zygmund strong law of large numbers, see for example Theorem 5.23 in \cite{Kallenberg2021}.
\begin{lemma}\label{lem:zygmund_SLLN}
	Suppose that $\E[W^2]<\infty$. For all $\gamma,\vartheta>1$ with $\gamma/\vartheta\leq 2$, 
	$$
		n^{-\vartheta}\sum_{i=1}^nW_i^{\gamma}\overset{a.s.}{\longrightarrow}0\quad\text{as}\quad n\to\infty.
		$$
\end{lemma}
We also employ the following inequality from Lemma 8 in \cite{lower_bound}, which can be proven via induction.
\begin{lemma}\label{lem:inequality_product}
	 For $k\in\N$ and $b,x_1,\dots,x_k\in(0,\infty)$ one has
	\begin{align*}
		\prod_{i=1}^k\frac{1}{b+x_i}\geq \frac{1}{b^k}-\frac{\sum_{i=1}^kx_i}{b^{k+1}}.
	\end{align*}
\end{lemma}
\begin{proof}[Proof of Theorem \ref{thm:vague_conv}]
	As mentioned before Theorem \ref{thm:vague_conv}, the claimed convergence is equivalent to convergence of the finite-dimensional distributions. Let thus $A\subseteq\N_{\geq 3}$ be finite. By the Poisson Cramér-Wold device, Corollary \ref{cor:cramer_wold}, it suffices to show for $q_k\in[0,1]$ for $k\in A$ that 
	$$
		\sum_{k\in A}\cC_n^{(q_k)}(k)=\sum_{k\in A}\sum_{\alpha \in I_k}X_{\alpha}^{(q_k)}\overset{d}{\longrightarrow}Z\sim \mathrm{Poi}\bigg(\sum_{k\in A}q_k\lambda_k\bigg)\quad\text{as}\quad n\to\infty.
	$$
	We use the notation from Section \ref{sec:setting}, $\cA=\cW$ and the choice for $B_\alpha$ in \eqref{def:B_alpha_conditioned}. Note that $X_\alpha^{(q_k)}$ is still a Bernoulli random variable so that the results from the previous section are applicable. By Lemma \ref{lem:Tv_bound_both_terms} it suffices to show that $b_1,b_2,b_3$ and 
	\begin{align}
		\bigg|\E_\cW\bigg[\sum_{k\in A}\cC_n^{(q_k)}(k)\bigg]-\sum_{k\in A}q_k\lambda_k\bigg|\label{eq:remaining_term}
	\end{align}
	converge in probability to zero as $n\to\infty$. As argued after \eqref{def:B_alpha_conditioned} we have $b_3=0$.  With 
	\begin{align*}
		\p_\cW(x\leftrightarrow y)=\frac{W_xW_y}{L_n+W_xW_y}\leq \frac{W_xW_y}{L_n}
	\end{align*}
	for distinct $x,y\in [n]$ we conclude that 
	\begin{align*}
		\E_\cW[X_\alpha]=\E_\cW\bigg[\prod_{i=1}^k \one\{\alpha_i\leftrightarrow\alpha_{i-1}\}\bigg]\leq \prod_{i=1}^k\frac{W_{\alpha_i}^2}{L_n}
	\end{align*}
	for $\alpha\in I_k$ and $k\in A$, where we wrote $\alpha_0=\alpha_k$. We obtain
	\begin{align*}
		&b_1=\sum_{k,\ell\in A} \sum_{\alpha \in I_k}\sum_{\beta\in B_\alpha\cap I_\ell}\E_\cW[X_\alpha^{(q_k)}]\E_\cW[X_\beta^{(q_\ell)}]\leq \sum_{k,\ell\in A} \sum_{\alpha \in I_k}\sum_{\beta\in B_\alpha\cap I_\ell}\E_\cW[X_\alpha]\E_\cW[X_\beta] \\
		&\leq\sum_{k,\ell\in A}\sum_{\alpha\in I_k}\sum_{\beta\in I_\ell}\one\{\alpha\text{ and } \beta\text{ share an edge}\}\prod_{i=1}^k\frac{W_{\alpha_i}^2}{L_n}\prod_{j=1}^\ell\frac{W_{\beta_i}^2}{L_n},
	\end{align*}
	by the definition of $B_\alpha$. If $\alpha\in I_k$ and $\beta\in I_\ell$ share an edge, the weights of the two adjacent vertices appear with a fourth power in the product above. Since all cycles belong to exactly one element of $I$, we may assume without loss of generality that $(\alpha_1,\alpha_2)=(\beta_1,\beta_2)$ by changing the enumerations of $\alpha$ and $\beta$.  Writing $L_n^{[\gamma]}=\sum_{i=1}^nW_i^\gamma$, we obtain
	\begin{align*}
		b_1&\leq \sum_{k,\ell\in A}\sum_{\alpha_1,\alpha_2\in [n]}\frac{W_{\alpha_1}^4W_{\alpha_2}^4}{L_n^4}\sum_{\alpha_3,\dots,\alpha_k\in[n]}\sum_{\beta_3,\dots,\beta_\ell\in [n]}\prod_{i=3}^k \frac{W_{\alpha_i}^2}{L_n}\prod_{j=3}^\ell \frac{W_{\beta_j}^2}{L_n}\\
		&=\left(\frac{L_n^{[4]}}{L_n^2}\right)^2\sum_{k,\ell\in A}\left(\frac{L_n^{[2]}}{L_n}\right)^{k+\ell-4}=\bigg(\frac{n^{-2}L_n^{[4]}}{\big(n^{-1}L_n\big)^2}\bigg)^2\sum_{k,\ell\in A}\left(\frac{L_n^{[2]}}{L_n}\right)^{k+\ell-4}.
	\end{align*}
	The sum is almost surely bounded due to the strong law of large numbers and $\E[W^2]<\infty$. Additionally, we apply Lemma \ref{lem:zygmund_SLLN} to obtain that the numerator of the first fraction tends almost surely to zero as $n\to\infty$. The strong law of large numbers yields that the denominator converges almost surely to $\E[W]^2$ as $n\to\infty$. Therefore, the whole fraction converges almost surely to zero as $n\to\infty$. In the following, we will simply refer to Lemma \ref{lem:zygmund_SLLN} for expressions similar to the first factor on the left-hand side of the equation above. For
	\begin{align}
		b_2=\sum_{k,\ell\in A}\sum_{\alpha\in I_k}\sum_{\beta\in B_\alpha\cap I_\ell\setminus\{\alpha\}}\E_\cW[X_\alpha^{(q_k)}X_\beta^{(q_\ell)}]\leq \sum_{k,\ell\in A}\sum_{\alpha\in I_k}\sum_{\beta\in B_\alpha\cap I_\ell\setminus\{\alpha\}}\E_\cW[X_\alpha X_\beta] \label{eq:bound_b_2}
	\end{align}
	we need to account for dependencies between $X_\alpha$ and $X_\beta$. For $\alpha\in I_k$ and $\beta\in B_\alpha\cap I_\ell\setminus\{\alpha\}$ we look at the graph union $\alpha\cup\beta$, where we take the union of the respective vertex and edge sets. This graph is given by $\alpha$ and some additional connections between vertices of $\alpha$ induced by $\beta$. We call these connections arcs. They are paths whose two endpoints belong to $\alpha$ and whose edges and inner points belong to $\beta$, but not to $\alpha$. We denote the number of such arcs by $m=m(\alpha,\beta)$, which is bounded by $k/2$ and $\ell/2$ as both endpoints of an arc belong to $\alpha$ and $\beta$. The numbers of edges of the arcs are denoted by $i_1,\dots,i_m$, where we enumerate the arcs in the order that one passes through them according to the labelling of $\beta$. It follows that $1\leq i_1,\dots,i_m\leq \ell-1$. Note that we can reconstruct $\beta$ from $\alpha$ and the arcs so that we may sum over all possibilities for the arcs instead of summing over the respective choices for $\beta$. We use that the existence of the edges of $\alpha$ and of the arcs are, conditionally on the weights, independent. The conditional expectation of the number of paths having $i$ edges with two fixed endpoints $\alpha_s$ and $\alpha_t$ is bounded by
	\begin{align*}
		\frac{W_{\alpha_s}W_{\alpha_t}}{L_n}\bigg(\frac{L_n^{[2]}}{L_n}\bigg)^{i-1}\leq \frac{W_{(n)}^2}{L_n}\bigg(\frac{L_n^{[2]}}{L_n}\bigg)^{i-1},
	\end{align*}
	due to summing over all choices for the $i-1$ vertices in between and writing $W_{(n)}=\max_{i\in[n]}W_i$. By summing over possible lengths, the contribution of all arcs is thus bounded by
	\begin{align*}
		 Y_{k,\ell}\coloneqq \sum_{m=1}^{\min(k,\ell)/2}(k)_{2m}\bigg(\frac{W_{(n)}^2}{L_n}\bigg)^{m}\sum_{i_1,\dots,i_m=1}^{\ell-1}\prod_{j=1}^{m}\bigg(\frac{L_n^{[2]}}{L_n}\bigg)^{i_j-1},
	\end{align*}
	where the factor $(k)_{2m}=k\cdot\hdots\cdot(k-2m+1)$ accounts for the possibilities to choose the endpoints of the arcs among $\alpha$. We find for $\varepsilon>0$ that
	$$
		\p(W_{(n)}^2>n\varepsilon)\leq n\p(W^2>n\varepsilon)\to0\quad \text{as}\quad n\to\infty,
	$$
	since $\E[W^2]<\infty$ implies $\p(W^2>t)=o(1/t)$ for $t\to\infty$. Therefore, 
	$$
		\frac{W_{(n)}^2}{L_n}=\frac{W_{(n)}^2}{n}\frac{n}{L_n}\overset{\p}{\longrightarrow}0
	$$
	as $n\to\infty$ by the weak law large numbers. This shows that $Y_{k,\ell}$ converges in probability to $0$ as $n\to\infty$ for all $k,\ell\in A$ as all sums are finite. From \eqref{eq:bound_b_2} we derive
	\begin{align*}
		b_2&\leq \sum_{k,\ell\in A}\sum_{\alpha\in I_k}\E_\cW[X_\alpha] Y_{k,\ell} \leq\sum_{k,\ell\in A}\sum_{\alpha\in[n]_{\neq}^k}\prod_{j=1}^k\frac{W_{\alpha_j}^2}{L_n}Y_{k,\ell}\leq \sum_{k,\ell\in A}\bigg(\frac{L_n^{[2]}}{L_n}\bigg)^kY_{k,\ell}\overset{\p}{\longrightarrow}0
	\end{align*}
	as $n\to\infty$ since $A$ is a finite set.
	
	For the remaining term in \eqref{eq:remaining_term} we use the triangle inequality to deduce
		\begin{align*}
		\bigg|\E_\cW\bigg[\sum_{k\in A}\cC_n^{(q_k)}(k)\bigg]-\sum_{k\in A}q_k\lambda_k\bigg|\leq \sum_{k\in A} \big|\E_\cW\big[\cC_n^{(q_k)}(k)\big]-q_k\lambda_k\big|.
	\end{align*}
	We show for fixed $k\in A$ that the summand on the right-hand side converges almost surely to zero as $n\to\infty$.  We have
	\begin{align*}
		\E_\cW\big[\cC_n^{(q_k)}(k)\big]=\E_\cW\bigg[\sum_{\alpha\in I_k}X_\alpha^{(q_k)}\bigg]=q_k\sum_{\alpha\in I_k}\E_\cW\big[X_\alpha\big]=\frac{q_k}{2k}\sum_{\alpha\in[n]^k_{\neq}}\prod_{i=1}^k\frac{W_{\alpha_i^2}}{L_n+W_{\alpha_i}W_{\alpha_{i-1}}},
	\end{align*}
	where the factor $2k$ makes up for starting vertex and orientation when labelling a fixed cycle and we write $\alpha_0=\alpha_k$ as usual. We obtain 
	\begin{align}
		\frac{2k}{q_k}\E_\cW\big[\cC_n^{(q_k)}(k)\big]=\sum_{\alpha\in[n]^k}\prod_{i=1}^k\frac{W_{\alpha_i}^2}{L_n+W_{\alpha_i}W_{\alpha_{i+1}}}-\sum_{\alpha\in[n]^k}\one\{\exists i\neq j\colon \alpha_i=\alpha_j\}\prod_{i=1}^k\frac{W_{\alpha_i}^2}{L_n+W_{\alpha_i}W_{\alpha_{i+1}}},\label{eq:cond_exp_cn}
	\end{align}
	where the absolute value of the second summand is bounded by 
	\begin{align*}
		\binom{k}{2}\frac{L_n^{[4]}}{L_n^2}\left(\frac{L_n^{[2]}}{L_n}\right)^{k-2}
	\end{align*}
	 by choosing two of the $k$ vertices that shall be equal. By Lemma \ref{lem:zygmund_SLLN} this converges almost surely to zero as $n\to\infty$. The first summand on the right hand side in \eqref{eq:cond_exp_cn} is upper bounded by 
	\begin{align*}
		\sum_{\alpha\in[n]^k}\prod_{i=1}^k\frac{W_{\alpha_i}^2}{L_n}=\left(\frac{L_n^{[2]}}{L_n}\right)^k\overset{a.s.}{\longrightarrow}\bigg(\frac{\E[W^2]}{\E[W]}\bigg)^k=2k\lambda_k.
	\end{align*}
	For a lower bound of the first summand on the right hand side in \eqref{eq:cond_exp_cn} we use Lemma \ref{lem:inequality_product} with $b=L_n$ and $x_i=W_{\alpha_i}W_{\alpha_{i-1}}$. This yields the lower bound
	\begin{align*}
		\sum_{\alpha\in[n]^k}\prod_{i=1}^k\frac{W_{\alpha_i}^2}{L_n}-\sum_{\alpha\in[n]^k}\frac{\prod_{i=1}^kW_{\alpha_i}^2\sum_{j=1}^kW_{\alpha_j}W_{\alpha_{j-1}}}{L_n^{k+1}}=\left(\frac{L_n^{[2]}}{L_n}\right)^k-k\left(\frac{L_n^{[3]}}{L_n^{3/2}}\right)^2\left(\frac{L_n^{[2]}}{L_n}\right)^{k-2}.
	\end{align*}
	As above the last summand goes almost surely to zero as $n\to\infty$ whereas the first summand converges almost surely to $2k\lambda_k$ as $n\to\infty$. Altogether we obtain from \eqref{eq:cond_exp_cn} that
	\begin{align*}
		\E_\cW\big[\cC_n^{(q_k)}(k)\big]\overset{a.s.}{\longrightarrow}q_k \lambda_k
	\end{align*}
	as $n\to\infty$, which concludes the proof.
\end{proof}

\section{Proofs of the quantitative theorems}\label{sec:quantitative}
Recall that we use the trivial $\sigma$-field for $\cA$ in this section, i.e.\ $\E_\cA$ reduces to expectation so that e.g.\ $p_\alpha=\E[X_\alpha]$. Also, contrary to the previous section, we use \eqref{def:B_alpha} as definition for $B_\alpha$. 
\subsection{Technical lemmas}
In Section \ref{sec:setting}, we defined for $k,\ell\in\N_{\geq 3}$ and $\alpha\in I_k,\beta\in I_\ell$ the quantities $p_\alpha$ and $p_{\alpha\beta}$, see \eqref{def:b1},\eqref{def:b2} and \eqref{eq:def_X_alpha}. Since $p_\alpha$, the probability of the existence of the cycle $\alpha$ having $k$ vertices, only depends on $k$ (and $n$, which we suppress in our notation) but not on the exact choice of $\alpha$, we may write $p_k=p_\alpha$. The term $p_{\alpha\beta}$, the probability for the existence of the two cycles $\alpha$ and $\beta$ of lengths $k$ and $\ell$, respectively, in turn depends on the relation of $\alpha$ and $\beta$. If they share edges, we obtain dependencies between the existence of $\alpha$ and the existence of $\beta$, even after conditioning on the weights. Therefore, we need to investigate the structure of $\alpha\cap\beta$. For a cycle $\alpha\in I_k$ we refer to its indices here and in the following always modulo $k$, e.g.\ $\alpha_{k+2}=\alpha_2$. We introduce segments, which correspond to the parts of $\alpha$ and $\beta$ consisting of consecutively shared vertices. For $k,\ell\geq3$ and two cycles $\alpha=(\alpha_1,\dots,\alpha_k)$ and $\beta=(\beta_1,\dots,\beta_\ell)$ let $i\in[k],j\in[\ell]$ be such that $\alpha_i=\beta_j$ but $\alpha_{i-1}\neq\beta_{j-1},\beta_{j+1}$. Now take $m\in[k]$ maximal with the property $\alpha_{i+1}=\beta_{j+1},\dots,\alpha_{i+m}=\beta_{j+m}$ or $\alpha_{i+1}=\beta_{j-1},\dots,\alpha_{i+m}=\beta_{j-m}$ so that $\alpha_i,\dots,\alpha_{i+m}$ are consecutively shared vertices of $\alpha$ and $\beta$. We call $(\alpha_i,\dots,\alpha_{i+m})$ a segment of length $m+1$. Note that segments are maximal in the sense that they cannot be extended in either direction. For instance, the two cycles $\alpha=(1,2,3,4,5,6)$ and $\beta=(2,3,6,7,1)$ share exactly two segments, the segment $(1,2,3)$ of length 3 and the segment $(6)$ of length 1. For two cycles $\alpha\in I_k,\beta\in I_\ell$ that intersect each other in $s$ segments of lengths $i_1,\dots,i_s$ we write $p_{k,\ell,s,i}=p_{\alpha\beta}$. This is well-defined. Finally, we denote the falling factorial by $(n)_k=n\cdot\hdots\cdot(n-k+1)$ for $n,k\in\N$. 
\begin{lemma}\label{lem:TV_bound}
	Let $n\in\N$ and $A\subseteq\{3,\dots,n\}$. For $Z\sim\mathrm{Poi}\big(\sum_{k\in A}\lambda_k\big)$,
	\begin{align*}
		d_{TV}(\cC_n(A),Z)&\leq \frac{1}{2n}\sum_{k,\ell\in A}p_kp_\ell n^{k+\ell}+\sum_{k,\ell\in A}\sum_{s=1}^k\sum_{i\in [k]^s}p_{k,\ell,s,i}(2k\ell)^{s-1}n^{k+\ell-|i|}\\
		&\quad +\sum_{k\in A}\bigg|\frac{(n)_k}{2k}p_k-\lambda_k\bigg|.
	\end{align*}
\end{lemma}
\begin{proof}
	We have
	\begin{align*}
		\E[\cC_n(A)]=\sum_{k\in A}\sum_{\alpha\in I_k}\E[X_\alpha]=\sum_{k\in A}\sum_{\alpha\in I_k}\E[X_\alpha]=\sum_{k\in A}\frac{(n)_k}{2k}p_k,
	\end{align*}
	where we used $|I_k|=(n)_k/(2k)$. Thus, Lemma \ref{lem:Tv_bound_both_terms} yields
	\begin{align*}		
		d_{TV}(\cC_n(A),Z)\leq b_1+b_2+b_3+\sum_{k\in A}\bigg|\frac{(n)_k}{2k}p_k-\lambda_k\bigg|.
	\end{align*}
	We first note that $b_3=0$ by the definition of $B_\alpha$ in \eqref{def:B_alpha} and the brief comment thereafter. It remains to bound $b_1$ and $b_2$. We compute
	\begin{align*}
		b_1&= \sum_{k,\ell\in A}\sum_{\alpha\in I_k} \sum_{\beta\in B_\alpha\cap I_\ell}\E[X_\alpha] \E[X_\beta]=\sum_{k,\ell\in A} p_k p_\ell \sum_{\alpha\in I_k} |B_\alpha\cap I_\ell|.
	\end{align*}
		To estimate the cardinalities of the sets, recall that $B_\alpha$ contains all cycles that intersect $\alpha\in I_k$ in at least one vertex. We have $k$ choices for a vertex of $\alpha$ which fixes one vertex of $\beta\in B_\alpha\cap I_\ell$. For the remaining ones, there are fewer than $n^{\ell-1}$ many choices. With $|I_k|\leq n^k/(2k)$ we derive
	\begin{align*}
		b_1\leq \sum_{k,\ell\in A}p_kp_\ell \frac{n^k}{2k}kn^{\ell-1}=\frac{1}{2n}\sum_{k,\ell\in A}p_kp_\ell n^{k+\ell},
	\end{align*}
	which is the first summand on the right-hand side of the claimed inequality. Additionally,
	\begin{align*}
		b_2&= \sum_{k,\ell\in A}\sum_{\alpha\in I_k}\sum_{\beta\in I_\ell\cap B_\alpha\setminus\{\alpha\}}\E[X_\alpha X_\beta]\leq 2\sum_{k,\ell\in A\colon k\leq \ell}\sum_{\alpha\in I_k}\sum_{\beta\in I_\ell\cap B_\alpha\setminus\{\alpha\}}\E[X_\alpha X_\beta].
	\end{align*}
	The value of $\E[X_\alpha X_\beta]$ depends on a more thorough analysis of the intersections of $\alpha$ and $\beta$. Recall the notion of a segment introduced above. Two distinct cycles $\alpha$ and $\beta$ may share several segments, i.e.\ there are maximal sequences of vertices $v_1\leftrightarrow v_2\leftrightarrow \hdots \leftrightarrow v_i$ connected by edges which are contained in the two cycles $\alpha$ and $\beta$. For $k,\ell\in A$ with $k\leq \ell$, $\alpha\in I_k$ and $\beta\in B_\alpha\cap I_\ell\setminus\{\alpha\}$ there can be at most $s=1,\dots,k$ of these segments. The case $s=0$ is excluded since the two cycles share at least one vertex. We denote the number of vertices of the segments by $1\leq i_1,\dots,i_s\leq k$ and write $i=(i_1,\dots,i_s)$. We define $$M(s,i,\alpha)=\bigg\{\beta\in \bigcup_{\ell\in A}I_\ell\colon \alpha\text{ and }\beta \text{ share } s \text{ segments of lengths }i_1,\dots,i_s\text{ in that order}\bigg\},$$ where we order the segments with respect to the enumeration of $\alpha$. We derive 
	\begin{align*}
		b_2&\leq 2 \sum_{k,\ell\in A\colon k\leq \ell}\sum_{\alpha\in I_k}\sum_{\beta\in I_\ell\cap B_\alpha\setminus\{\alpha\}}\sum_{s=1}^k\sum_{i\in[k]^s}\one\{\beta\in M(s,i,\alpha)\}\E[X_\alpha X_\beta]\\
		&\leq  2\sum_{k,\ell\in A\colon k\leq \ell}\sum_{\alpha\in I_k}\sum_{s=1}^k\sum_{i\in[k]^s}|B_\alpha \cap I_\ell \cap M(s,i,\alpha)|p_{k,\ell,s,i}.
	\end{align*}
	For $k,\ell\in A$ with $k\leq \ell$ as well as $\alpha\in I_k,s\in[k]$ and $i\in[k]^s$ we bound the cardinality of $B_\alpha\cap I_\ell\cap M(s,i,\alpha)$. We write $|i|=\sum_{j=1}^si_j$ for the total number of shared vertices. Since $\alpha=(\alpha_1,\dots,\alpha_k)\in I_k$ has a fixed order, we may talk about the first vertex of a segment, i.e.\ the endpoint $\alpha_j$ of the segment with minimal $j$ among the two endpoints. There are at most $k^s$ possibilities to choose the first vertices of the $s$ segments among the $k$ vertices of $\alpha$, call them $v_1,\dots,v_s$. We may assume without loss of generality that $v_1$ is the first entry of $\beta=(\beta_1,\dots,\beta_\ell)$. This leaves us with placing the $s-1$ remaining vertices $v_2,\dots,v_s$ among the $\ell-1$ remaining slots in $\beta$, giving us fewer than $\ell^{s-1}$ choices. We need to decide whether $v_j$ is the first or the last vertex of the $j$-th segment with respect to the orientation of $\beta$. This gives us another factor $2^{s-1}$ because we may assume that the orientation of $\beta$ coincides with the orientation of the first segment. This determines the $|i|$ entries of $\beta$ that also belong to $\alpha$. Finally, we fill the remaining $\ell-|i|$ entries of $\beta$, amounting to fewer than $n^{\ell-|i|}$ choices. In total we obtain
	\begin{align*}
		b_2&\leq 2\sum_{k,\ell\in A\colon k\leq \ell}\sum_{\alpha\in I_k}\sum_{s=1}^k\sum_{i\in[k]^s}k^s\ell^{s-1}2^{s-1}n^{\ell-|i|}p_{k,\ell,s,i}\\
		&\leq 2\sum_{k,\ell\in A}\frac{n^k}{2k}\sum_{s=1}^k\sum_{i\in[k]^s}k^s\ell^{s-1}2^{s-1}n^{\ell-|i|}p_{k,\ell,s,i}=\sum_{k,\ell\in A}\sum_{s=1}^k\sum_{i\in[k]^s}(2k\ell)^{s-1}n^{k+\ell-|i|}p_{k,\ell,s,i},
	\end{align*}
	which concludes the proof.	
\end{proof}
 The next lemma is a special case of Lemma 1 in \cite{bobkov}. It is some variant of the Chernoff bound for random variables with finite second moments.
\begin{lemma}\label{lem:An_exp_decay}
	Let $n\in\N$ and $Y_1,\dots,Y_n$ be non-negative i.i.d.\ random variables with $\E[Y_1^2]<\infty$. For all $0<\lambda<1$ there exists a constant $c>0$ depending only on $\lambda$ and the first and second moment of $Y$ such that
	\begin{align*}
		\p\bigg(\sum_{i=1}^nY_i\leq \lambda \E[Y_1]n\bigg)\leq\exp\left(-c n\right).
	\end{align*}
\end{lemma}

The following three lemmas are used for the three summands in the bound of Lemma \ref{lem:TV_bound} and contain parts (a) and (b). We will use part (a) for Theorem \ref{thm:conv_rate_bounded}, i.e.\ it corresponds to considering finitely many cycle lengths. Part (b) will be used for Theorem \ref{thm:conv_rate_unbounded} and allows the considered cycles to grow logarithmically in $n$. We start with uniform bounds on $p_k$.
\begin{lemma}\label{lem:bound_p_alpha}
	Suppose that $\E[W^2]<\infty$.
	\begin{enumerate}[label=(\alph*)]
		\item Let $N\in\N$. There exists a constant $C>0$ such that for all $n\in\N,3\leq k\leq N$,
		\begin{align*}
			p_k\leq Cn^{-k}.
		\end{align*}
		
		\item Let $a>0$ and $\E[W^2]<\E[W]$. There exists a constant $C>0$ such that for all $n\in\N$ and $3\leq k\leq \lfloor a\log(n)\rfloor $,
		\begin{align*}
			p_k\leq Cn^{-k}.
		\end{align*}
	\end{enumerate}
	
\end{lemma}
\begin{proof}
	We provide a general bound, which we use for both claims. Let $3\leq k\leq n,\lambda\in(0,1),\alpha\in I_k$ and
	$$A_n=\{L_n>\lambda\E[W]n\}.$$ 
	Since $X_\alpha=\one\{\alpha_1\leftrightarrow \alpha_2\leftrightarrow\hdots\leftrightarrow \alpha_k\leftrightarrow \alpha_1\}$ is bounded by one we have 
	\begin{align*}
		p_k=p_\alpha=\E[X_\alpha]\leq \p(A_n^c)+\E[\one_{A_n}X_\alpha].
	\end{align*}
	By Lemma \ref{lem:An_exp_decay} we know that $\p(A_n^c)$ decays faster than
	$$
	n^{-a\log(n)}=\exp(-a\log(n)^2),
	$$
	which is the fastest rate claimed in the lemma. Therefore, it suffices to show that 
	$$
	\E[\one_{A_n}X_\alpha]\leq Dn^{-k}
	$$
	for some constant $D$ in the respective setting. We write $\alpha_0=\alpha_k$. The conditional independence of distinct edges and the definition of $A_n$ yield
	\begin{align}
		&\E[\one_{A_n}X_\alpha]=\E\bigg[\one_{A_n}\E_\cW\bigg[\prod_{i=1}^k\one\{\alpha_i\leftrightarrow\alpha_{i-1}\}\bigg]\bigg]=\E\bigg[\one_{A_n}\prod_{i=1}^{k}\p_\cW(\alpha_i\leftrightarrow \alpha_{i-1})\bigg]\label{eq:bd_X_alpha}\\
		&=\E\bigg[\one_{A_n}\prod_{i=1}^k\frac{W_{\alpha_i}W_{\alpha_{i-1}}}{L_n+W_{\alpha_i}W_{\alpha_{i-1}}}\bigg]\leq \E\bigg[\prod_{i=1}^k\frac{W_i^2}{\lambda\E[W]n}\bigg]=\bigg(\frac{\E[W^2]}{\lambda\E[W]}\bigg)^kn^{-k}.\nonumber
	\end{align}
	The first claim follows by taking $$D=\max_{k=3,\dots,N}\bigg(\frac{\E[W^2]}{\lambda\E[W]}\bigg)^k.$$ 	For the second claim we choose $\lambda=\E[W^2]/\E[W]$, which is smaller than $1$ by assumption, so that $D=1$.
\end{proof}
The next lemma is used to bound $p_{k,\ell,s,i}$ in the second summand of Lemma \ref{lem:TV_bound}. 
\begin{lemma}\label{lem:bound_p_alpha_beta}
	Suppose that $\E[W^4]<\infty$.
	\begin{enumerate}[label=(\alph*)]
		\item Let $N\in\N$. There exists a constant $C>0$ such that for all $n\in\N,3\leq k, \ell\leq N$ and $\alpha\in I_k,\beta\in I_\ell$ which share $0\leq j\leq \min(k,\ell)$ edges we have
		\begin{align*}
			p_{\alpha\beta}\leq Cn^{-k-\ell+j}.
		\end{align*}
		\item Let $a>0$ and $\E[W^2]<\E[W]$. There exist constants $C>0,\kappa\geq 1$ such that for all $n\in\N$ and  $3\leq k,\ell\leq \lfloor a\log(n) \rfloor$ and distinct $\alpha\in I_k,\beta\in I_\ell$ which intersect each other in $1\leq s\leq k$ segments of lengths $1\leq i_1,\dots,i_s\leq \min(k,\ell)$,
		\begin{align*}
			p_{k,\ell,s,i}\leq C \kappa^s n^{-k-\ell+|i|-s},
		\end{align*}
		where $|i|=\sum_{j=1}^si_j$ denotes the number of all vertices that $\alpha$ and $\beta$ share.
		
	\end{enumerate}
\end{lemma}	
\begin{proof}
	For fixed $n\in\N$ and $\lambda\in(0,1)$ we define
	$$A_n=\{L_n>\lambda\E[W]n\}.$$  We obtain for all $3\leq k,\ell\leq n$ and $\alpha\in I_k,\beta\in I_\ell$ that
	\begin{align*}
		p_{\alpha\beta}=\E[X_\alpha X_\beta]\leq \p(A_n^c)+\E[\one_{A_n}X_\alpha X_\beta].
	\end{align*}			
	The rates of convergence claimed in the lemma are all slower than
	$$ n^{-3a\log(n)}= \exp(-3a\log(n)^2).$$	By Lemma \ref{lem:An_exp_decay} $\p(A_n^c)$ decays faster. Therefore, the claims follow if one can show
	$$
	\E[\one_{A_n}X_\alpha X_\beta]\leq Dn^{-k-\ell+j}\quad\text{and}\quad \E[\one_{A_n}X_\alpha X_\beta]\leq D\kappa^sn^{-k-\ell+|i|-s}
	$$
	for some constant $D>0$, respectively. Write $\alpha\cup\beta$ for the graph union of the cycles induced by $\alpha$ and $\beta$. We denote its vertices by $V(\alpha\cup\beta)$, its edges by $E(\alpha\cup\beta)$ and degree of $v\in V(\alpha\cup\beta)$ in $\alpha\cup\beta$ by $\deg_{\alpha\cup\beta}(v)$. Similarly, we refer to the set of edges that belong to both $\alpha$ and $\beta$ as $E(\alpha\cap\beta)$. We derive
	\begin{align*}
		&\E[\one_{A_n}X_\alpha X_\beta]=\E\bigg[\one_{A_n}\prod_{\{x,y\}\in E(\alpha\cup\beta)}\one\{x\leftrightarrow y\}\bigg]=\E\bigg[\one_{A_n}\prod_{\{x,y\}\in E(\alpha\cup\beta)}\p_\cW(x\leftrightarrow y)\bigg]\\
		&\leq \E\bigg[\one_{A_n}\prod_{\{x,y\}\in E(\alpha\cup\beta)}\frac{W_xW_y}{L_n}\bigg]\leq(\lambda\E[W]n)^{-|E(\alpha\cup\beta)|}\prod_{x\in V(\alpha\cup\beta)}\E\bigg[W_x^{\deg_{\alpha\cup\beta}(x)}\bigg].
	\end{align*}
	We continue with showing the first claim. Suppose that there is a fixed $N\in\N$ with $k,\ell\leq N$ and that $\alpha$ and $\beta$ share $j$ edges. Then $|E(\alpha\cup\beta)|=k+\ell-j\leq 2N$ and the degree of any vertex $v\in V(\alpha\cup\beta)$ must lie in $\{2,3,4\}$. Since $|V(\alpha\cup\beta)|\leq 2N$,
	\begin{align*}
		\E[\one_{A_n}X_\alpha X_\beta]&\leq \max\bigg(1,(\lambda\E[W])^{-2N}\bigg)n^{-k-\ell+j}\max(1,\E[W^2],\E[W^3],\E[W^4])^{2N}\eqqcolon Dn^{-k-\ell+j},
	\end{align*}
	which shows the first claim.
	
	For the second claim, one can choose $\lambda=\E[W^2]/\E[W]<1$ since $\E[W^2]<\E[W]$. Let $a>0$,  $k,\ell\leq \lfloor a\log(n) \rfloor$ and $\alpha\in I_k$ and $\beta\in I_\ell$ share $s$ segments which contain $i_1,\dots,i_s\geq 1$ vertices. Therefore, $|V(\alpha\cup\beta)|=k+\ell-|i|$. Now we analyse how many vertices of the possible degrees $2,3$ and $4$ are in $\alpha\cup\beta$. Whenever a segment has length one, this vertex needs to have degree $4$. We write $m_i(1)$ for the number of segments of length $1$  so that we have $m_i(1)$ vertices of degree $4$ in $\alpha\cup\beta$. When a segment has at least length two, there will be two vertices of degree $3$, namely start- and endpoint of the segment. Here we used that $\alpha\neq\beta$. We obtain $2(s-m_i(1))$ vertices of degree $3$. The remaining $k+\ell-|i|-2s+m_i(1)$ vertices have degree $2$. Finally, the number of edges is given by $|E(\alpha\cup\beta)|=k+\ell-|i|+s$. We obtain
	\begin{align*}
		&\E[\one_{A_n}X_\alpha X_\beta]\leq (n\E[W^2])^{-k-\ell+|i|-s}\E[W^4]^{m_i(1)}\E[W^3]^{2(s-m_i(1))}\E[W^2]^{k+\ell-|i|-2s+m_i(1)}\\
		&=\bigg(\frac{\E[W^4]\E[W^2]}{\E[W^3]^2}\bigg)^{m_i(1)}\bigg(\frac{\E[W^3]^2}{\E[W^2]^3}\bigg)^sn^{-k-\ell+|i|-s}\eqqcolon \gamma^{m_i(1)}\delta^{s} n^{-k-\ell+|i|-s}\leq \kappa^sn^{-k-\ell+|i|-s},
	\end{align*}
	where $\kappa=\max(1,\gamma)\delta$ and the last inequality uses $m_i(1)\leq s$. This concludes the proof.
\end{proof}
The following lemma will be employed to bound the third summand in Lemma \ref{lem:TV_bound}. 
\begin{lemma}\label{lem:bound_comparing_exp}
	Suppose that $\E[W^3]<\infty$.
	\begin{enumerate}[label=(\alph*)]
		\item Let $N\in\N$. There exists a constant $C>0$ such that for all $3\leq k\leq N$ and $n\in\N$,
		\begin{align*}
			\bigg|\frac{(n)_k}{2k}p_k-\lambda_k\bigg|\leq \frac{kC}{n}.
		\end{align*}
		\item Let $a>0$ and $\E[W^2]<\E[W]$. There exists a constant $C>0$ such that for all $n\in\N$ and $3\leq k\leq \lfloor a\log(n) \rfloor$,
		\begin{align*}
			\bigg|\frac{(n)_k}{2k}p_k-\lambda_k\bigg|\leq \frac{kC}{n}.
		\end{align*}
	\end{enumerate}
\end{lemma}
\begin{proof}
	We provide a single proof that works for both claims with the correct parameter choices. Let $M_n=N$ for the first claim and $M_n=\lfloor a\log(n)\rfloor$ for the second claim and $3\leq k\leq M_n$. We may assume without loss of generality that $M_n< n$ as $M_n=o(n)$ as $n\to\infty$. With $W_0=W_k$, 
	\begin{align*}
		p_k=\E\bigg[\prod_{i=1}^k\frac{W_i^2}{L_n+W_iW_{i-1}}\bigg].
	\end{align*}
	We provide lower and upper bounds on $p_k$ to obtain the desired statements. For $\lambda\in(0,1)$ and $1\leq j<n$ we define $L_{n,j}=\sum_{i=j+1}^nW_i$ and
	\begin{align*}
		A_n= \{L_{n,M_n}>(n-M_n)\lambda\E[W]\},
	\end{align*}
	where we choose $\lambda=1/2$ for the first claim of the lemma and $\lambda=\E[W^2]/\E[W]$ for the second claim. Note that $A_n$ is independent of $W_1,\dots,W_{k}$. For an upper bound we compute 
	\begin{align*}
		p_k\leq \E\bigg[\prod_{i=1}^k\frac{W_i^2}{L_n}\bigg]=\E\bigg[\one_{A_n}\prod_{i=1}^k\frac{W_i^2}{L_n}\bigg]+\E\bigg[\one_{A_n^c}\prod_{i=1}^k\frac{W_i^2}{L_n}\bigg]\eqqcolon \E\bigg[\one_{A_n}\prod_{i=1}^k\frac{W_i^2}{L_n}\bigg]+\frac{2k}{(n)_k}R_1(k).
	\end{align*}
	For a lower bound, we use Lemma \ref{lem:inequality_product} with $b=L_n$ and $x_i=W_iW_{i-1}$ for $i=1,\dots,k$ so that
	\begin{align*}
		p_k&=\E\left[\prod_{i=1}^k\frac{W_i^2}{L_n+W_iW_{i-1}}\right]\geq \E\left[\one_{A_n}\prod_{i=1}^k\frac{W_i^2}{L_n+W_iW_{i-1}}\right]\\
		&\geq \E\left[\one_{A_n}\prod_{i=1}^k\frac{W_i^2}{L_n}\right]-\E\left[\one_{A_n}\frac{1}{L_n^{k+1}}\sum_{j=1}^kW_jW_{j-1}\prod_{i=1}^kW_i^2\right]\eqqcolon \E\left[\one_{A_n}\prod_{i=1}^k\frac{W_i^2}{L_n}\right]-\frac{2k}{(n)_k}R_2(k).
	\end{align*}		
	We continue by showing
	\begin{align}
		\bigg|\frac{(n)_k}{2k}\E\bigg[\one_{A_n}\prod_{i=1}^k\frac{W_i^2}{L_n}\bigg]-\lambda_k\bigg|\leq\sum_{i=3}^6R_i(k)\label{eq:bound_R_i}
	\end{align}
	with
	\begin{align*}
		R_3(k)&=\frac{(n)_k}{2k}\E\bigg[\one_{A_n}\bigg(\frac{1}{L_{n,k}^k}-\frac{1}{L_{n}^k}\bigg)\prod_{i=1}^kW_i^2\bigg],\\
		R_4(k)&=(n)_k\lambda_k\E\bigg[\one_{A_n}\bigg(\frac{1}{L_{n,k}^k}-\frac{1}{L_n^k}\bigg)\prod_{i=1}^kW_i\bigg],\\
		R_5(k)&=(n)_k\lambda_k\E\bigg[\one_{A_n^c}\prod_{i=1}^k\frac{W_i}{L_n}\bigg]\text{ and}\\
		R_6(k)&=\lambda_k\E\bigg[\sum_{v\in[n]^k_{=}}\prod_{i=1}^k\frac{W_{v_i}}{L_n}\bigg],
	\end{align*}
	where $[n]^k_{=}=[n]^k\setminus[n]^k_{\neq}$ denotes the $k$-tuples in $[n]^k$ with at least two equal entries.	After establishing \eqref{eq:bound_R_i}, combining the lower and upper bound yields
	\begin{align*}
		\bigg|\frac{(n)_k}{2k}p_k-\lambda_k\bigg|\leq \sum_{i=1}^6R_i(k)
	\end{align*} 
	and it remains to show the claimed bound for $R_1(k),\dots,R_6(k)$. We show \eqref{eq:bound_R_i} by rewriting
	\begin{align*}
		\frac{(n)_k}{2k}\E\bigg[\one_{A_n}\prod_{i=1}^k\frac{W_i^2}{L_n}\bigg]=\frac{(n)_k}{2k}\E\bigg[\one_{A_n}\prod_{i=1}^k\frac{W_i^2}{L_{n,k}}\bigg]-R_3(k).
	\end{align*}
	Because $A_n$ and $L_{n,k}$ are independent of $W_1,\dots,W_k$, we derive
	\begin{align}
		&\frac{(n)_k}{2k}\E\bigg[\one_{A_n}\prod_{i=1}^k\frac{W_i^2}{L_{n,k}}\bigg]=\frac{(n)_k}{2k}\prod_{i=1}^k\frac{\E[W_i^2]}{\E[W_i]}\times \E\bigg[\one_{A_n}\prod_{j=1}^k\frac{W_j}{L_{n,k}}\bigg]=(n)_k\lambda_k\E\bigg[\one_{A_n}\prod_{i=1}^k\frac{W_i}{L_{n,k}}\bigg]\nonumber\\
		&=(n)_k\lambda_k\E\bigg[\one_{A_n}\prod_{i=1}^k\frac{W_i}{L_{n}}\bigg]+R_4(k)=(n)_k\lambda_k\E\bigg[\prod_{i=1}^k\frac{W_i}{L_n}\bigg]-R_5(k)+R_4(k).\nonumber
	\end{align}
	Observe that $|[n]^k_{\neq}|=(n)_k$ and use $\sum_{v\in[n]^k}\prod_{i=1}^kW_{v_i}=L_n^k$ to rewrite
	\begin{align*}
		(n)_k\lambda_k\E\bigg[\prod_{i=1}^k\frac{W_i}{L_{n}}\bigg]&=\lambda_k\E\bigg[\sum_{v\in[n]^k_{\neq}}\prod_{i=1}^k\frac{W_{v_i}}{L_{n}}\bigg]=\lambda_k-\lambda_k\E\bigg[\sum_{v\in[n]^k_{=}}\prod_{i=1}^k\frac{W_{v_i}}{L_{n}}\bigg]=\lambda_k-R_6(k).
	\end{align*}		
	Applying the triangle inequality shows \eqref{eq:bound_R_i} and we proceed with establishing 
	\begin{align}
		R_i(k)\leq \frac{kC_i}{n} \label{eq:bound_Rik}
	\end{align}
	for some constant $C_i>0$ and $i=1,\dots,6$. From Lemma \ref{lem:An_exp_decay} we obtain some $c>0$ with
	\begin{align*}
		n^{2k+1}\p(A_n^c)\leq n^{2M_n+1}\exp(-c(n-M_n))=\exp((2M_n+1)\log(n)-cn+cM_n)\to 0
	\end{align*}
	as $n\to\infty$. In particular there exists a constant $c_A$ not depending on $k$ and $n$ such that
	\begin{align*}
		\p(A_n^c)\leq c_An^{-2k-1}.
	\end{align*}
	Using $W_i\leq L_n$ for $i=1,\dots,n$ as well as the independence of $A_n$ and $W_1,\dots,W_{M_n}$ we get
	\begin{align*}
		R_1(k)=\frac{(n)_k}{2k}\E\bigg[\one_{A_n^c}\prod_{i=1}^k\frac{W_i^2}{L_n}\bigg]\leq n^k\E[W]^k\p(A_n^c)\leq c_A\E[W]^kn^{k-2k-1}= c_A\bigg(\frac{\E[W]}{n}\bigg)^kn^{-1}
	\end{align*}
	showing \eqref{eq:bound_Rik} for $i=1$.
	Concerning $R_5(k)$ we obtain similarly, using the definition of $\lambda_k$ in \eqref{eq:def_lambda_k},
	\begin{align}
		R_5(k)=(n)_k\lambda_k\E\bigg[\one_{A_n^c}\prod_{i=1}^k\frac{W_i}{L_n}\bigg]\leq n^k\bigg(\frac{\E[W^2]}{\E[W]}\bigg)^k\p(A_n^c)\leq \bigg(\frac{\E[W^2]}{n\E[W]}\bigg)^kc_An^{-1},\label{eq:R5}
	\end{align}
	which implies \eqref{eq:bound_Rik} for $i=5$. Additionally, we have 
	\begin{align}
		R_6(k)=\lambda_k\E\bigg[\one_{A_n}\sum_{v\in[n]^k_{=}}\prod_{i=1}^k\frac{W_{v_i}}{L_n}\bigg]+\lambda_k\E\bigg[\one_{A_n^c}\sum_{v\in[n]^k_{=}}\prod_{i=1}^k\frac{W_{v_i}}{L_n}\bigg]. \label{eq:R6_bounds}
	\end{align}
	For the second summand, note that the sum is bounded by one so that
	\begin{align*}
		\lambda_k\E\bigg[\one_{A_n^c}\sum_{v\in[n]^k_{=}}\prod_{i=1}^k\frac{W_{v_i}}{L_n}\bigg]\leq \frac{1}{2k}\bigg(\frac{\E[W^2]}{\E[W]}\bigg)^k\p(A_n^c),
	\end{align*}
	which can be bounded as in \eqref{eq:R5}. For the first summand on the right hand side in \eqref{eq:R6_bounds}, there are fewer than $k^2$ choices for the indices of two equal entries of $v\in[n]^k_=$. This gives us
	\begin{align*}
		\sum_{v\in[n]^k_=}\prod_{i=1}^kW_{v_i}\leq k^2\sum_{v\in[n]}W_v^2L_n^{k-2}
	\end{align*}
	and thus, using $A_n$ to bound $L_n$,
	\begin{align*}
		\lambda_k\E\bigg[\one_{A_n}\sum_{v\in[n]^k_{=}}\prod_{i=1}^k\frac{W_{v_i}}{L_{n}}\bigg]\leq \lambda_k k^2n\E\bigg[\one_{A_n}\frac{W_1^2}{L_n^2}\bigg]\leq \frac{k}{2n}\bigg(\frac{n}{n-M_n}\bigg)^2\frac{\E[W^2]^{k+1}}{\lambda^2\E[W]^{k+2}}.
	\end{align*}
	The second factor on the right-hand side converges to $1$ as $n\to\infty$ since $M_n=o(n)$. The last term is bounded by its maximum over $k=3,\dots,N$ in the first case. In the second case we have $\E[W^2]<\E[W]$ so that the fraction is bounded by $\lambda^{-2}/\E[W]$, showing \eqref{eq:bound_Rik} for $i=6$. 
	
	 Before we continue with the remaining terms $R_2(k),R_3(k)$ and $R_4(k)$, we observe that the inequality $1+x\leq e^x$ for $x\in\R$ and $M_n^2=o(n)$ imply
	\begin{align}
		1\leq \bigg(\frac{n}{n-M_n}\bigg)^{k}\leq \bigg(\frac{n}{n-M_n}\bigg)^{M_n}=\bigg(1+\frac{M_n}{n-M_n}\bigg)^{M_n}\leq \exp\bigg(\frac{M_n^2}{n-M_n}\bigg)\to 1\label{eq:asympt_eq_n-log(n)_n}
	\end{align}
	as $n\to\infty$. For $R_2(k)$ we again use $A_n$ to bound $L_n^{-k-1}$ and compute
	\begin{align}
		R_2(k)&=\frac{(n)_k}{2k}\E\bigg[\one_{A_n}\frac{1}{L_n^{k+1}}\sum_{j=1}^kW_jW_{j-1}\prod_{i=1}^kW_i^2\bigg]\leq \frac{n^k}{2k}\frac{k\E[W^3]^2\E[W^2]^{k-2}}{((n-M_n)\lambda\E[W])^{k+1}}\nonumber\\
		&=\frac{1}{2n}\bigg(\frac{n}{n-M_n}\bigg)^{k+1}\frac{\E[W^3]^2\E[W^2]^{k-2}}{(\lambda\E[W])^{k+1}}.\label{eq:R2_bound}
	\end{align}
	By \eqref{eq:asympt_eq_n-log(n)_n}, the second fraction in \eqref{eq:R2_bound} is bounded by some constant. For the third fraction we do a case distinction. For $3\leq k\leq N$ the bound is trivial by taking the maximum. In the second setting, plugging in $\lambda=\E[W^2]/\E[W]$ simplifies the expression to $\E[W^3]^2/\E[W^2]^3$ and shows \eqref{eq:bound_Rik} for $i=2$.	For $R_3(k)$ we apply the mean value theorem to $f(x)=x^{-k}$. This gives us
	\begin{align*}
		\frac{1}{L_{n,k}^k}-\frac{1}{L_n^k}\leq k\sum_{j=1}^kW_jL_{n,k}^{-k-1}
	\end{align*}
	so that, using $A_n$ to bound $L_{n,k}^{-k-1}$,
	\begin{align*}
		R_3(k)&=\frac{(n)_k}{2k}\E\bigg[\one_{A_n}\bigg(\frac{1}{L_{n,k}^k}-\frac{1}{L_{n}^k}\bigg)\prod_{i=1}^kW_i^2\bigg]\leq \frac{n^k}{2k}\E\bigg[\one_{A_n}k\sum_{j=1}^kW_jL_{n,k}^{-k-1}\prod_{i=1}^kW_i^2\bigg]\\
		&\leq \frac{n^k}{2k}k^2\E[W^3]\E[W^2]^{k-1}\bigg(\frac{1}{(n-M_n)\lambda\E[W]}\bigg)^{k+1}=\frac{k}{2n}\bigg(\frac{n}{n-M_n}\bigg)^{k+1}\frac{\E[W^3]\E[W^2]^{k-1}}{(\lambda\E[W])^{k+1}}.
	\end{align*}
	The same argument as for \eqref{eq:R2_bound} yields \eqref{eq:bound_Rik} for $i=3$. Similarly, we obtain for $R_4(k)$ that
	\begin{align*}
		R_4(k)&=(n)_k\lambda_k\E\bigg[\one_{A_n}\bigg(\frac{1}{L_{n,k}^k}-\frac{1}{L_n^k}\bigg)\prod_{i=1}^kW_i\bigg]\\
		&\leq \frac{n^k}{2k}\bigg(\frac{\E[W^2]}{\E[W]}\bigg)^kk^2\E[W^2]\E[W]^{k-1}\bigg(\frac{1}{(n-M_n)\lambda\E[W]}\bigg)^{k+1}\\
		&= \frac{k}{2n}\bigg(\frac{n}{n-M_n}\bigg)^{k+1}\bigg(\frac{\E[W^2]}{\E[W]}\bigg)^{k+1}\frac{1}{\E[W]\lambda^{k+1}}.
	\end{align*}
	This concludes the proof.
\end{proof}
\subsection{Proofs of the theorems}\label{subsec:quantitative_theorems}
\begin{proof}[Proof of Theorem \ref{thm:conv_rate_bounded}]	
	By Lemma \ref{lem:TV_bound} we have for $A\subseteq\{3,\dots,N\}$,
	\begin{align*}
		d_{TV}(\cC_n(A),Z)&\leq \frac{1}{2n}\sum_{k,\ell=3}^Np_kp_\ell n^{k+\ell}+\sum_{k,\ell=3}^N\sum_{s=1}^k\sum_{i\in [k]^s}p_{k,\ell,s,i}(2k\ell)^{s-1}n^{k+\ell-|i|}\\
		&\quad +\sum_{k=3}^N\bigg|\frac{(n)_k}{2k}p_k-\lambda_k\bigg|.
	\end{align*}
	We show that all summands can be bounded by $C/n$ for some $C>0$. For the third summand, part (a) of Lemma \ref{lem:bound_comparing_exp} yields the existence of some $C_1>0$ such that for all $n\in\N$,
	\begin{align*}
		\sum_{k=3}^N\bigg|\frac{(n)_k}{2k}p_k-\lambda_k\bigg|\leq \sum_{k=3}^N \frac{kC_1}{n}\leq \frac{C_1N^2}{n}.
	\end{align*}
	  Next we adress the first summand. From part (a) of Lemma \ref{lem:bound_p_alpha} we know that there exists $C_2>0$ with $p_k\leq C_2n^{-k}$ for all $3\leq k\leq N$ and $n\in\N$. Therefore,
	\begin{align*}
		\frac{1}{2n}\sum_{k,\ell=1}^Np_kp_\ell n^{k+\ell}\leq \frac{N^2C_2^2}{2n}.
	\end{align*}
	For the second summand, note that when $\alpha$ and $\beta$ intersect in $s$ segments of lengths $i_1,\dots,i_s$ they share $|i|-s$ edges. Thus, part (a) of Lemma \ref{lem:bound_p_alpha_beta} yields the existence of some $C_3>0$ with
	\begin{align*}
		&\sum_{k,\ell=1}^N\sum_{s=1}^k\sum_{i\in [k]^s}p_{k,\ell,s,i}(2k\ell)^{s-1}n^{k+\ell-|i|}\leq \sum_{k,\ell=1}^N\sum_{s=1}^k\sum_{i\in [k]^s}C_3n^{-k-\ell+|i|-s}(2k\ell)^{s-1}n^{k+\ell-|i|}\\
		&=\frac{C_3}{n} \sum_{k,\ell=1}^N \sum_{s=0}^{k-1} k^{s+1}(2kl)^{s}n^{-s},
	\end{align*}
	where the sums are bounded in $n$, which concludes the proof.	
\end{proof}

\begin{proof}[Proof of Theorem \ref{thm:no_long_cycles}]
	By assumption we have $\E[W^2]<\E[W]$ so that we may choose
	\begin{align*}
		\lambda=\frac{\E[W^2]+\E[W]}{2\E[W]}<1.
	\end{align*}
	Define $A_n=\{L_n> \lambda\E[W]n\}$ so that Lemma \ref{lem:An_exp_decay} yields a positive constant $c$	such that
	\begin{align}
		\p(\cC_n(\lfloor a\log(n) \rfloor+1,\dots,n)>0)&\leq \p(A_n^c)+\p(\one_{A_n}\cC_n(\lfloor a\log(n) \rfloor+1,\dots,n)>0)\nonumber\\
		&\leq e^{-cn}+\sum_{k=\lfloor a\log(n) \rfloor+1}^n\sum_{\alpha\in I_k}\E[\one_{A_n}X_\alpha],\label{eq:no_cycles}
	\end{align}
	where we used the Markov inequality for the second summand. Since the first summand decays faster than the desired order, it suffices to show the claimed bound for the second summand. With \eqref{eq:bd_X_alpha} we compute for $k\geq3$ and $\alpha\in I_k$,
	\begin{align*}
		\E[\one_{A_n}X_\alpha]&\leq n^{-k}\bigg(\frac{\E[W^2]}{\E[W]}\frac{2\E[W]}{\E[W^2]+\E[W]}\bigg)^k=n^{-k}\bigg(\frac{2\E[W^2]}{\E[W^2]+\E[W]}\bigg)^k= n^{-k} p^k,
	\end{align*}
	with $p$ as in Theorem \ref{thm:no_long_cycles} and where $\E[W^2]<\E[W]$ yields $p<1$. Together with $|I_k|\leq n^k $ we obtain for the second summand in \eqref{eq:no_cycles} that
	\begin{align*}
		\sum_{k=\lfloor a\log(n) \rfloor+1}^n\sum_{\alpha\in I_k}\E[\one_{A_n}X_\alpha]\leq \sum_{k=\lfloor a\log(n) \rfloor+1}^np^k\leq p^{\lfloor a\log(n) \rfloor+1}\sum_{k=0}^\infty p^k=p^{\lfloor a\log(n) \rfloor+1}\frac{1}{1-p}.
	\end{align*}
	Since $\lfloor a\log(n) \rfloor+1\geq a\log(n)$ and $p<1$, we have
	\begin{align*}
		p^{\lfloor a\log(n) \rfloor+1}\leq p^{a\log(n)}= n^{\log(p)a},
	\end{align*}
	which proves the claim.
\end{proof}

\begin{proof}[Proof of Theorem \ref{thm:conv_rate_unbounded}]	
	Let
	\begin{align*}
		p=\frac{2\E[W^2]}{\E[W^2]+\E[W]} \quad \text{and}\quad \rho=\frac{\E[W^2]}{\E[W]}
	\end{align*}
	so that $p,\rho<1$ due to $\E[W^2]<\E[W]$. Choose $a>0$ large enough such that
	\begin{align*}
		a\log(p)<-1\quad \text{and}\quad a\log(\rho)<-1.
	\end{align*}
	Writing $M_a=\{3,4,\dots,\lfloor a\log(n)\rfloor\}$ and $\overline{M}_a=\{\lfloor a\log(n)\rfloor +1,\dots\}$ we have
	\begin{align}
		d_{TV}(\cC_n(A),\eta(A))&\leq d_{TV}(\cC_n(A\cap M_a),\eta(A\cap M_a))+\p(\cC_n(A)\neq\cC_n(A\cap M_a))\nonumber\\
		&\quad +\p(\eta(A)\neq \eta(A\cap M_a))\nonumber\\
		&=d_{TV}(\cC_n(A\cap M_a),\eta(A\cap M_a))+\p(\cC_n(\overline{M}_a)>0)+\p(\eta(\overline{M}_a)>0).\label{eq:Tv_bound_thm_2}
	\end{align}
	For the second summand above we use Theorem \ref{thm:no_long_cycles} to obtain a constant $C_1>0$ with
	\begin{align*}
		\p(\cC_n(\overline{M}_a)>0)\leq C_1n^{a\log(p)}\leq C_1n^{-1},
	\end{align*}
	by the choice of $a$. For the third summand the definition of $\eta$ yields
	\begin{align*}
		\p(\eta(\overline{M}_a)>0)=1-\p(\eta(\overline{M}_a)=0)=1-\exp\bigg(-\sum_{k=\lfloor a\log(n)\rfloor +1}^\infty \lambda_k\bigg).
	\end{align*}
		We use $1-\exp(-x)\leq x$ for $x\in\R$ and $\lambda_k=\rho^k/(2k)\leq \rho^k$ to derive
	\begin{align*}
		1-\exp\bigg(-\sum_{k=\lfloor a\log(n) \rfloor+1}^\infty \lambda_k\bigg)\leq \sum_{k=\lfloor a\log(n) \rfloor+1}^\infty \rho^k\leq \rho^{\lfloor a\log(n) \rfloor+1}\sum_{k=0}^\infty \rho^k.
	\end{align*}
	Since $\rho<1$, this is further bounded by
	\begin{align*}
		 \rho^{a\log(n)}\frac{1}{1-\rho}=n^{a\log(\rho)}\frac{1}{1-\rho}\leq \frac{n^{-1}}{1-\rho},
	\end{align*}
	again by the choice of $a$. It remains to consider the first summand of \eqref{eq:Tv_bound_thm_2}. By Lemma \ref{lem:TV_bound} we have
	\begin{align*}
		d_{TV}(\cC_n(A\cap M_a),\eta(A\cap M_a))&\leq \frac{1}{2n}\sum_{k,\ell\in M_a}p_kp_\ell n^{k+\ell}+\sum_{k,\ell\in M_a}\sum_{s=1}^k\sum_{i\in [k]^s}p_{k,\ell,s,i}(2k\ell)^{s-1}n^{k+\ell-|i|}\\
		&\quad +\sum_{k\in M_a}\bigg|\frac{(n)_k}{2k}p_k-\lambda_k\bigg|.
	\end{align*}
	For the third summand we apply part (b) of Lemma \ref{lem:bound_comparing_exp} which yields the existence of some constant $D_1>0$ such that for all $n\in\N$,
	\begin{align*}
		\sum_{k\in M_a}\bigg|\frac{(n)_k}{2k}p_k-\lambda_k\bigg|\leq \sum_{k=3}^{\lfloor a\log(n) \rfloor}\frac{D_1k}{n}\leq  \frac{D_1a^2\log(n)^2}{n}.
	\end{align*}
	For the first summand we use part (b) of Lemma \ref{lem:bound_p_alpha} to obtain the existence of a constant $D_2>0$ with
	\begin{align*}
	\frac{1}{2n}\sum_{k,\ell\in M_a}p_kp_\ell n^{k+\ell}\leq \frac{1}{2n}\sum_{k,\ell=3}^{\lfloor a\log(n) \rfloor}D_2^2n^{-k-\ell}n^{k+\ell}\leq \frac{D_2^2a^2\log(n)^2}{2n}.
	\end{align*}
	Finally, for the second summand, we obtain via part (b) of Lemma \ref{lem:bound_p_alpha_beta} the existence of constants $D_3>0$ and $\kappa\geq 1$ such that
	\begin{align*}
		&\sum_{k,\ell\in M_a}\sum_{s=1}^k\sum_{i\in [k]^s}p_{k,\ell,s,i}(2k\ell)^{s-1}n^{k+\ell-|i|}\leq \sum_{k,\ell=3}^{\lfloor a\log(n)\rfloor}\sum_{s=1}^k\sum_{i\in [k]^s}D_3\kappa ^sn^{-k-\ell+|i|-s}(2k\ell)^{s-1}n^{k+\ell-|i|}\\
		&=\sum_{k,\ell=3}^{\lfloor a\log(n) \rfloor}\frac{D_3}{2k\ell}\sum_{s=1}^k\bigg(\frac{2k^2\ell\kappa}{n}\bigg)^s= \sum_{k,\ell=3}^{\lfloor a\log(n) \rfloor}\frac{k\kappa D_3}{n}\sum_{s=0}^{k-1}\bigg(\frac{2k^2\ell\kappa}{n}\bigg)^s\\
		&\leq \frac{\kappa D_3}{n}a^3\log(n)^3\sum_{s=0}^n\bigg(\frac{2\lfloor a\log(n) \rfloor^3\kappa}{n}\bigg)^s,
		\end{align*}
		where the series converges to $1$ due to $\lfloor a\log(n) \rfloor^3/n\to0$ as $n\to\infty$ and a comparison to the geometric series. This shows the claim.
\end{proof}

\begin{proof}[Proof of Theorem \ref{thm:shortest_longest_cycle}]
	We relate the lengths of the shortest and longest cycle, respectively, to properties of the point process $\cC_n$. Additionally, we express the probabilities of the limiting random variables $\cS$ and $\cL$ in terms of the point process $\eta$. The claim then follows from Theorem \ref{thm:conv_rate_unbounded}.
	For the claim on the shortest cycle we compute
	\begin{align*}
		&d_{Kol}(\cC_{\min}^{(n)},\cS )=\sup_{t\in\R}|\p(\cC_{\min}^{(n)}\leq t)-\p(\cS \leq t)|\\
		&\leq |\p(\cC_{\min}^{(n)}=0)-\p(\cS =0)|+\sup_{t\in\N_{\geq3}}|\p(3\leq \cC_{\min}^{(n)}\leq t)-\p(3\leq \cS \leq t)|.
	\end{align*}
	We defined the shortest cycle to have length zero if and only if there is no cycle. Additionally, the shortest cycle having a length between 3 and $n$ is equivalent to the existence of at least one cycle of that length. With the definitions of $\eta$ and $\cS$ the term above equals
	\begin{align*}
		&|\p(\cC_n(3,4,\dots)=0)-\p(\eta(3,4,\dots)=0)|+\sup_{t\in\N_{\geq3}}\p(\cC_n(3,\dots,t)>0)-\p(\eta(3,\dots,t)>0)|\\
		&\leq 2\sup_{A\subseteq\N_{\geq 3}}d_{TV}\big(\cC_n(A),\eta(A)\big)\leq \frac{2C\log(n)^3}{n}
	\end{align*}
	for some constant $C>0$ by Theorem \ref{thm:conv_rate_unbounded}. For the longest cycle we proceed similarly and have
	\begin{align*}
		&d_{Kol}(\cC_{\max}^{(n)},\cL )=\sup_{t\in\R}|\p(\cC_{\max}^{(n)}\leq t)-\p(\cL \leq t)|\\
		&\leq |\p(\cC_{\max}^{(n)}=0)-\p(\cL =0)|+\sup_{t\in\N_{\geq3}}|\p(\cC_{\max}^{(n)}>t)-\p(\cL>t)|\\
		&=|\p(\cC_n(3,4,\dots)=0)-\p(\eta(3,4,\dots)=0)|\\
		&\quad +\sup_{t\in\N_{\geq3}}|\p(\cC_n(t+1,\dots)>0)-\p(\eta(t+1,\dots)>0)|\\
		&\leq 2\sup_{A\subseteq\N_{\geq 3}}d_{TV}\big(\cC_n(A),\eta(A)\big)\leq \frac{2C\log(n)^3}{n},
	\end{align*}
	which concludes the proof.
\end{proof}
\section{Proofs of Lemma \ref{lem:two_mom_suff_conditioned_version} and Lemma \ref{lem:Tv_bound_both_terms}}\label{sec:poisson_lemma}

\begin{proof}[Proof of Lemma \ref{lem:two_mom_suff_conditioned_version}.]
	The total variation distance satisfies
	\begin{align*}
		d_{TV}(S,T_\cA)=\frac{1}{2}\sup_{||h||_{\infty}=1}|\E[h(S)-h(T_\cA)]|,
	\end{align*}
	where the supremum runs over all functions $h\colon\N_0\to\R$ with supremum norm equal to one. Thus,
	\begin{align*}
		2d_{TV}(S,T_\cA)&=\sup_{||h||_{\infty}=1}\big|\E\big[\E_\cA[h(S)-h(T_\cA)]\big]\big|\leq \E\bigg[\sup_{||h||_{\infty}=1}\big|\E_\cA[h(S)-h(T_\cA)]\big|\bigg]\\
		&= \E\bigg[\min\bigg(2,\sup_{||h||=1}\E_\cA[h(S)-h(T_\cA)]\bigg)\bigg],
	\end{align*}
	where the last equality uses $||h||_{\infty}=1$ and omits the absolute value as one can choose $-h$ when the expression is negative. Now one can mimic the proof of Theorem 1 in \cite{twomomentssuffice} to bound the supremum and obtain the desired statement.
\end{proof}
	
\begin{proof}[Proof of Lemma \ref{lem:Tv_bound_both_terms}]
		Let $T_\cA \sim\mathrm{Poi}(\E_\cA[S])$, then the triangle inequality and Lemma \ref{lem:two_mom_suff_conditioned_version} yield
		\begin{align*}
			d_{TV}(S,Z)\leq d_{TV}(S,T_\cA)+d_{TV}(T_\cA,Z)\leq \E\big[\min(1,b_1+b_2+b_3)\big]+d_{TV}(T_\cA,Z).
		\end{align*}
	We compute
		\begin{align*}
			&d_{TV}(T_\cA,Z)=\sup_{A\subseteq \N_0}|\p(T_\cA\in A)-\p(Z\in A)|=\sup_{A\subseteq \N_0}|\E[\p_\cA(T_\cA \in A)-\p(Z\in A)]|\nonumber\\
			&\leq \E[\sup_{A\subseteq \N_0}|\p_\cA(T_\cA\in A)-\p(Z\in A)|]=\E[\min(1,\sup_{A\subseteq \N_0}|\p_\cA(T_\cA\in A)-\p(Z\in A)|)].
		\end{align*}
		Since, conditionally on $\cA$, $T_\cA$ follows a Poisson distribution with parameter $\lambda_\cA=\E_\cA[S]$ and $Z$ follows a Poisson distribution with parameter $\mu$, we obtain for all $A\subseteq\N_0$,
		\begin{align*}
			|\p_\cA(T_\cA\in A)-\p(Z\in A)|\leq|\lambda_\cA-\mu|,
		\end{align*}
		see e.g.\ Example 1 in Section 3 of Chapter 21 in \cite{schinazi}. The assertion follows.
\end{proof}
\subsection*{Acknowledgements}
The author would like to thank Matthias Schulte for fruitful discussions and valuable comments on the presentation of the results. Additional thanks go to Vanessa Trapp for helpful feedback on an earlier draft.

\end{document}